\documentclass[12pt,twoside]{article}
\usepackage[all]{}
\usepackage {amssymb,latexsym,amsthm,amscd,amsmath}
\newtheorem{theorem}{Theorem}[section]
\newtheorem{lemma}{Lemma}[section]
\newtheorem{corollary}{Corollary}[section]
\newtheorem{proposition}{Proposition}[section]

\newtheorem*{theorem*}{Theorem}
\theoremstyle{definition}

\numberwithin{equation}{subsection}

\newcommand{\ignore}[1]{}

\newcommand{\mynote}[1]{}
\begin{document}
\setcounter{section}{0}
% document information
\title{\bf On $R$-triviality of $F_4$-II}
%\author{Maneesh Thakur}
\author{Maneesh Thakur }
%\keywords{Automorphisms, Albert algebras, Structure group, inner structure group, Kneser-Tits}
\date{}
\maketitle
%\subjclass{2010 Math. Subject Classification. Primary 20G15, Secondary 17C30}
\begin{abstract}
\noindent
\it{Simple algebraic groups of type $F_4$ defined over a field $k$ are the full automorphism groups of Albert algebras over $k$. Let $A$ be an Albert algebra over a field $k$ of arbitrary characteristic whose all isotopes are isomorphic. We prove that $\text{\bf Aut}(A)$ is $R$-trivial, in the sense of Manin. If $k$ contains cube roots of unity and $A$ is any Albert algebra over $k$, we prove that there is an isotope $A^{(v)}$ of $A$ such that $\text{\bf Aut}(A^{(v)})$ is $R$-trivial.} 
\end{abstract} 

%%%%%%%%%%%%%%%%%%%%%%%%%%%%%%%%%%%%%%%%%%%
\section{\bf Introduction}
This work builds on the results from (\cite{Th-5}). We recall at this stage that simple algebraic groups of type $F_4$ defined over a field $k$ are precisely the full groups of automorphisms of Albert algebras defined over $k$. In papers (\cite{Th-5}, \cite{A-C-P-1}), the authors proved (independently) that the automorphism group of an Albert algebra arising from the first Tits construction is $R$-trivial, the proof in (\cite{Th-5}) being characteristic free.

In this paper, we prove that for an Albert algebra $A$ over a field $k$ of arbitrary characteristic whose all isotopes are isomorphic, the algebraic group $\text{\bf Aut}(A)$ is $R$-trivial. Since for first Tits construction Albert algebras all isotopes are isomorphic (see \cite{PR2}), results in this paper improve the result proved in (\cite{A-C-P-1}, \cite{Th-5}).

We remark here that the class of Albert algebras whose all isotopes are isomorphic strictly contains the class of Albert algebras that are first Tits constructions (see \cite{PR2}). The case of reduced Albert algebras which have no nonzero nilpotents and Albert division algebras in general is work in progress.

The knowledge that a group is $R$-trivial helps in studying its rationality properties. Recently, in the papers (\cite{Th-4}, \cite{A-C-P-2}), $R$-triviality of the structure group of Albert division algebras was proved, thereby proving the longstanding Tits-Weiss conjecture for Albert division algebras as well as the Kneser-Tits conjecture for groups of type $E^{78}_{8,2}$ and $E^{78}_{7,1}$. These groups have the structure group of an Albert division algebra as their anisotropic kernel. In (\cite{Th-1}) we had proved $R$-triviality for $\text{\bf Aut}(A)$ for pure first Tits construction Albert algebras, while in (\cite{Th-2}), $R$-triviality of the structure group was proved to settle the Tits-Weiss conjecture for first Tits construction Albert division algebras, as well as reduced Albert algebras. In the paper (\cite{P-T-W}), a proof of Kneser-Tits conjecture for groups of type $E^{66}_{8,2}$ is given by proving $R$-triviality of such groups. The fundamental result of P. Gille (\cite{G}) connects the group of $R$-equivalent points of an isotropic, simple, simply connected algebraic group with its Whitehead group, thereby reducing the Kneser-Tits problem to proving $R$-triviality of the group or otherwise.

In this paper, we prove that, for an Albert algebra $A$ over a field $k$ of arbitrary characteristic whose all isotopes are isomorphic, the group $\text{\bf Aut}(A)$ is $R$-trivial. We get the $R$-triviality of the group of automorphisms of first Tits construction Albert algebras as a consequence. If $k$ contains cube roots of unity and $A$ is any Albert algebra over $k$, we prove that there exists $v\in A$ such that $\text{\bf Aut}(A^{(v)})$ is $R$-trivial.

\section{\bf Preliminaries} We refer the reader to (\cite{P2}), (\cite{SV}) and (\cite{KMRT}) for basic material on Albert algebras. In this section, we quickly introduce some notions related to Albert algebras that are indispensable. All base fields will be 
assumed to be infinite of arbitrary characteristic unless specified otherwise. \\
\noindent
For what follows, we refer to (\cite{PR7} or \cite{PR5}). Let $J$ be a finite dimensional vector space over a field $k$. A \emph{cubic norm structure} on $J$ is a triple $(N,\#, c)$, $N:J\rightarrow k$ is a cubic form called the \emph{norm}, $\#:J\rightarrow J$ is a quadratic map called the \emph{adjoint} and $c\in J$ is a base point, called the \emph{identity element} of the norm structure. There is an associated trace bilinear form $T:J\times J\rightarrow k$ and these satisfy some identities (see \cite{PR7} or \cite{PR5}).
\iffalse%%%%%%%%%%%%%%%%%%%%%%%%%%%%%%%%%%%%%%%%%%%%%%%%%%%%%%%%%%%%%%%%
\begin{enumerate} 
\item $N:J\rightarrow k$ is a \emph{cubic form} on $J$,
\item $N(c)=1$,
\item the \emph{trace form} $T:J\times J\rightarrow k$, defined by $T(x,y):=-\Delta^x_c\Delta^y log N$, is nondegenerate, 
\item the \emph{adjoint} map $\#:J\rightarrow J$ defined by $T(x^{\#},y)=\Delta^y_x N$, is a quadratic map such that 
\item $x^{\#\#}=N(x)x$,
\item $c^{\#}=c$,
\item $c\times x=T(x)c-x$, where $T(x):=T(x,c)$ is the \emph{trace} of $x$ and $x\times y:=(x+y)^{\#}-x^{\#}-y^{\#}$,
\end{enumerate}
and these conditions hold in all scalar base changes of $J$. Here $\Delta^y_xf$ is the \emph{directional derivative} of a polynomial function $f$ on $J$, in the direction $y$, evaluated at $x$ and $\Delta^y log f$ denotes the \emph{logarithmic derivative} of $f$ in the direction $y$. For details, we refer to (\cite{J1}, Chap. VI, Sect. 2).
%%%%%%%%%%%%%%%%%%%%%%%%%%%%%%%%%%%%%%%%%%%%%%%%%%%%%%%%%%%%%%%%%%%%%%%%
\fi
Let $x\in J$. Define  
$$U_x(y):=T(x,y)x-x^{\#}\times y,~y\in J,$$
where $a\times b:=(a+b)^{\#}-a^{\#}-b^{\#},a,b\in J$. 
Then with $1_J:=c$ and the endomorphisms $U_x$ as $U$-operators, $J$ is a unital \emph{quadratic Jordan algebra} (see \cite{McK}), denoted by $J(N,c)$. An element $x\in J$ is defined to be \emph{invertible} if $N(x)\neq 0$ and $x^{-1}:=N(x)^{-1}x^{\#}$. The structure $J(N,c)$ is a \emph{division algebra} if $U_x$ is surjective for all $x\neq 0$, or equivalently, $N(x)\neq 0$ for all $x\neq 0$. Special Jordan algebras of degree $3$ provide imporant class of examples, we list them below for our purpose: \\
\vskip1mm
\noindent
{\bf Example.} Let $D$  be a separable associative algebra over $k$ of degree $3$. Let $N_D$ denote its norm and $T_D$ the trace. Let $\#:D\rightarrow D$ be the adjoint map. Then $(N_D, \#, 1_D)$ is a cubic norm structure, where $1_D$ is the unit element of $D$. This yields a quadratic Jordan algebra structure on $D$, which we will denote by $D_+$. \\
\noindent
 Let $(B,\sigma)$ be a separable associative algebra over $k$ with an involution $\sigma$ of the second kind (over its center). With the unit element $1$ of $B$ and the restriction of the norm $N_B$ of $B$ to $(B,\sigma)_+:=\{b\in B|\sigma(b)=b\}$, we obtain a cubic norm structure and hence a Jordan algebra structure on $(B,\sigma)_+$ which is a substructure of $B_+$. \\
\noindent
\vskip1mm
\noindent
    {\bf Tits processes and Tits constructions :} Let $D$ be a finite dimensional associative $k$-algebra of degree $3$ with norm $N_D$ and trace $T_D$. Let $\lambda\in k^{\times}$. On the $k$-vector space
    $D\oplus D\oplus D$, we define a cubic norm structure as below.  
$$1:=(1,0,0),~N((x,y,z)):=N_D(x)+\lambda N_D(y)+\lambda^{-1}N_D(z)-T_D(xyz),$$
$$(x,y,z)^{\#}:=(x^{\#}-yz,\lambda^{-1}z^{\#}-xy,\lambda y^{\#}-zx).$$  
    The Jordan algebra associated to this norm structure is denoted by $J(D,\lambda)$. The algebra $D_+$ is a subalgebra of $J(D,\lambda)$, corresponding to the first summand.
    The algebra $J(D,\lambda)$ is a division algebra if and only if $\lambda\notin N_D(D)$ (see \cite{PR7}, 5.2). This construction is called the \emph{first Tits process} arising from the parameters $D$ and $\lambda$. 

    Let $K$ be a quadratic \'{e}tale extension of $k$ and $B$ a separable associative algebra of degree $3$ over $K$ with an involution $\sigma$ over $K/k$ . Let $x\mapsto \overline{x}$ denote
    the nontrivial $k$-automorphism of $K$. For an \emph{admissible pair} $(u,\mu)$, i.e., $u\in (B,\sigma)_+$ such that $N_B(u)=\mu\overline{\mu}$ for some $\mu\in K^{\times}$, define a cubic norm structure on the $k$-vector space $(B,\sigma)_+\oplus B$ as follows:
$$N((b,x)):=N_B(b)+T_K(\mu N_B(x))-T_B(bxu\sigma(x)),$$
$$(b,x)^{\#}:=(b^{\#}-xu\sigma(x), \overline{\mu}\sigma(x)^{\#}u^{-1}-bx),~1:=(1_B,0).$$ 
The Jordan algebra obtained from this cubic norm structure is denoted by $J(B,\sigma,u,\mu)$. Note that $(B,\sigma)_+$ is a subalgebra of $J(B,\sigma, u,\mu)$ through the first summand. Then $J(B,\sigma,u,\mu)$ is a division algebra if and only if $\mu$ is not a norm from $B$ (see \cite{PR7}, 5.2). This construction is called the \emph{second Tits process} arising from the parameters $(B,\sigma),u$ and $\mu$.
\iffalse%%%%%%%%%%%%%%%%%%%%%%%%%%%%%
When $K=k\times k$, then $B=D\times D^{\circ}$ and $\sigma$ is the switch involution, where $D^{\circ}$ denotes the opposite algebra of $D$. In this case, the second construction $J(B,\sigma,u,\mu)$ can be identified with a first construction $J(D,\lambda)$.
%%%%%%%%%%%%%%%%%%%
\fi
\vskip1mm
\noindent
The Tits process starting with a central simple algebra $D$ and $\lambda\in k^{\times}$ yields the \emph{first Tits construction} Albert algebra $A=J(D,\lambda)$ over $k$. 
Similarly, in the Tits process if we start with a central simple algebra $(B,\sigma)$ with center a quadratic \'{e}tale algebra $K$ over $k$ and an involution $\sigma$ of the second kind, $u,\mu$ as described above, we get the \emph{second Tits construction} Albert algebra $A=J(B,\sigma,u,\mu)$ over $k$. 
\noindent
One knows that all Albert algebras can be obtained via Tits constructions. 

An Albert algebra is a division algebra if and only if its (cubic) norm $N$ is anisotropic over $k$ (see \cite{KMRT}, \S 39).
If $A=J(B,\sigma,u,\mu)$ as above, then $A\otimes_kK\cong J(B,\mu)$ as $K$-algebras, where $K$ is the center of $B$ (see \cite{KMRT}, 39.4).
\vskip1mm
\noindent 
Let $A$ be an Albert algebra over $k$. If $A$ arises from  the first construction, but does not arise from the second construction then we call $A$ a
    \emph{pure first construction} Albert algebra. Similarly, \emph{pure second construction} Albert algebras are defined as those which do not arise from the first Tits construction. 
\vskip1mm
\noindent
For an Albert division algebra $A$, any subalgebra is either $k$ or a cubic subfield of $A$ or of the form $(B,\sigma)_+$ for a degree $3$ central simple algebra $B$ with an involution $\sigma$ of the second kind over its center $K$, a quadratic \'{e}tale extension of $k$ (see \cite{J1}, Chap. IX, \S 12, Lemma 2,  \cite{PR5}). 
\vskip1mm
\noindent
    {\bf Norm similarities, isotopes :} Let $J$ be a quadratic Jordan algebra of degree $3$ over $k$ and $N$ its norm map. By a \emph{norm similarity} of $J$ we mean a bijective $k$-linear map $f:J\rightarrow J$ such that $N(f(x))=\nu(f)N(x)$ for all $x\in J$ and some $\nu(f)\in k^{\times}$. When $k$ is infinite, the notions of norm similarity and isotopy for degree $3$ Jordan algebras coincide (see \cite{J1}, Chap. VI, Thm. 6, Thm. 7). Let $v\in J$ be invertible. The $k$-vector space $J$ with the new identity element $1^{(v)}:=v^{-1}$ and the $U$-operator defined by $U_x^{(v)}:= U_xU_p$, is called the $v$-isotope of $J$ and is denoted by $J^{(v)}$.  For a second Tits process $J(B,\sigma,u,\mu)$ as above and $v\in (B,\sigma)_+$ invertible, we have $J(B,\sigma, u, \mu)^{(v)}\cong J(B,\sigma_v, uv^{\#}, N(v)\mu)$ (see \cite{P2}, 3.11).
    
\noindent
Let $J$ be a degree $3$ Jordan algebra over $k$ with norm map $N$. For $a\in J$, the $U$-operator $U_a$ is given by $U_a(y):=T(a,y)a-a^{\#}\times y,~y\in J$. When $a\in J$ is invertible, one knows that $U_a$ is a norm similarity of $J$, in fact, for any $x\in A,~N(U_a(x))=N(a)^2N(x)$. It also follows that $N(x^{\#})=N(x)^2$. 
\vskip1mm
\noindent
{\bf Albert algebras and algebraic groups :} For a $k$-algebra $X$ and a field extension $L$ of $k$, $X_L$ will denote the $L$-algebra $X\otimes_kL$. Let $A$ be an Albert algebra over $k$ with norm $N$ and $\overline{k}$ be an algebraic closure of $k$. It is well known that the full group of automorphisms 
$\text{\bf Aut}(A):=\text{Aut}(A_{\overline{k}})$ is a simple algebraic group of type $F_4$ defined over $k$ and all simple groups of type $F_4$ defined over $k$ arise this way .
We will denote the group of $k$-rational points of $\text{\bf Aut}(A)$ by $\text{Aut}(A)$, i.e., $\text{Aut}(A)=\text{\bf Aut}(A)(k)$. It is known that $A$ is a division algebra if and only if the norm form $N$ of $A$ is anisotropic (see \cite{Spr}, Thm. 17.6.5). Albert algebras whose norm form is isotropic over $k$, i.e. has a nontrivial zero over $k$, are called \emph{reduced}. 

The \emph{structure group} of $A$ is the full group $\text{\bf Str}(A)$ of norm similarities of $N$, is a connected reductive group over $k$, of type $E_6$. We denote by $\text{Str}(A)$ the group of $k$-rational points $\text{\bf Str}(A)(k)$.

%The $U$-operators $U_a$, for $a\in A$ invertible, are norm similarities and generate a normal subgroup of $\text{Str}(A)$, called the \emph{inner structure group} of $A$,
%which we will denote by $\text{Instr}(A)$.
%The commutator subgroup $\text{\bf Isom}(A)$ of $\text{\bf Str}(A)$ is the full group of isometries of $N$ and is a simple, simply connected group of type $E_6$, anisotropic over $k$ if and only if $N$ %is anisotropic, if and only if $A$ is a division algebra (see \cite{SV}, \cite{Spr}).

The automorphism group $\text{\bf Aut}(A)$ is the stabilizer of $1\in A$ in $\text{\bf Str}(A)$.
 \vskip1mm
\noindent
%    {\bf Strongly inner groups of type $E_6$ :} Let $H$ be a simple, simply connected algebraic group defined and strongly inner of type $E_6$ over $k$ (see \cite{T1} for definition).
%Then $H$ is isomorphic to $\text{\bf Isom}(A)$, the (algebraic) group of all norm isometries of an Albert algebra $A$ over $k$.
 %   Moreover, $H$ is anisotropic if and only if $A$ is a division algebra (see \cite{Spr}, Thm. 17.6.5).  

\vskip1mm
\noindent
{\bf $R$-triviality of algebraic groups:}
Let $X$ be an irreducible variety over a field $k$ with the set of $k$-rational points $X(k)\neq \emptyset$. We call $x,y\in X(k)$ as $R$-equivalent 
if there exists a sequence $x_0=x, x_1,\cdots, x_n=y$ of points in $X(k)$ and rational maps $f_i:\mathbb{A}_k^1\rightarrow X,~1\leq i\leq n$, defined over $k$ and regular at $0$ and $1$, such that $f_i(0)=x_{i-1},~f_i(1)=x_i$ (see \cite{M}). 

Let $G$ be a connected algebraic group defined over $k$. The set of points in $G(k)$ that are $R$-equivalent to $1\in G(k)$ is a normal subgroup of $G(k)$, denoted by $RG(k)$. The set $G(k)/R$ of $R$-equivalence classes in $G(k)$ is in a canonical bijection with $G(k)/RG(k)$ and thus has a natural group structure. We identify $G(k)/R$ with the group $G(k)/RG(k)$. This group is useful in studying rationality properties of $G$. 

Call $G$ \emph{$R$-trivial} if $G(L)/R=\{1\}$ for all field extensions $L$ of $k$. A variety $X$ defined over $k$ is defined to be $k$-\emph{rational} if $X$ is birationally isomorphic over $k$ to an affine space. One knows that $G$ is $k$-rational then $G$ is $R$-trivial (see \cite{Vos}, Chap. 6, Prop. 2).  Also, owing to the fact that tori of rank at most $2$ are rational, one knows that algebraic groups of rank at most $2$ are rational. 

\section{\bf $R$-triviality results} In this section, we prove results on $R$-triviality of various groups and also prove our main theorems. We need a few results from (\cite{Th-2}) which we recall below. We remark here that if $A$ is a \emph{reduced} Albert algebra, then there is an isotope $A^{(v)}$ of $A$ such that $A^{(v)}$ contains nonzero nilpotents. Hence $\text{Aut}(A^{(v)})$ is a simple abstract group (Thm. 14, \cite{J-iv}). Since $R\text{\bf Aut}(A^{(v)})(k)$ is a normal subgroup of $\text{\bf Aut}(A^{(v)})(k)$ and $\text{\bf Aut}(A^{(v)})$ contains $k$-rational subgroups, it follows that $\text{\bf Aut}(A^{(v)})$ is $R$-trivial. This therefore proves that for any reduced Albert algebra $A$, whose all isotopes are isomorphic, the group $\text{\bf Aut}(A)$ is $R$-trivial. We record this as
\begin{proposition}\label{reduced-R-triv} Let $A$ be a reduced Albert algebra over a field $k$. Then there exists an invertible element $v\in A$ such that the algebraic group $\text{\bf Aut}(A^{(v)})$ is $R$-trivial.
\end{proposition}
\begin{corollary} Let $A$ be a reduced Albert algebra over a field $k$ whose all isotopes are isomorphic. Then $\text{\bf Aut}(A)$ is $R$-trivial.
\end{corollary}
\begin{proof} As we discussed above, there exists an isotope $A^{(v)}$ of $A$ which contains a nonzero nilpotent. For such an isotope, for any field extension $M$ of $k$ the algebra $A^{(v)}\otimes M=
  (A\otimes M)^{(v\otimes 1)}$ also contains nonzero nilpotents. Since all isotopes of $A$ are isomorphic, it follows that $A$ contains nonzero nilpotents and hence $A\otimes M$ contains nonzero nilpotents. Since the group of automorphisms of a reduced Albert algebra which contain nonzero nilpotents is an abstract simple group, it follows by the above discussion that $\text{Aut}(A\otimes M)=R \text{\bf Aut}(A)(M)$. Hence $\text{\bf Aut}(A)$ is $R$-trivial. 
  \end{proof}
Hence, we may assume for the rest of the paper that $A$ is a division Albert algebra. We need the following result (Prop. 3.1, \cite{A-C-P-2}):
%We begin with
%\begin{theorem}\label{R-triv} Let $A$ be an Albert division algebra. Then for any $9$-dimensional subalgebra $S$ of $A$, the subgroup $\text{\bf Aut}(A/S)$ of $\text{\bf Aut}(A)$, consisting of all automorphisms of $A$ fixing $S$ pointwise, is $R$-trivial.
%\end{theorem}
%\begin{proof} This follows immediately from Proposition \ref{rational}.   
%\end{proof} 
\begin{proposition}\label{aut-S-rtriv} Let $A$ be an Albert division algebra. Let $S\subset A$ be a $9$-dimensional subalgebra. Then, with the notations as above, $\text{\bf Aut}(A,S)$ is rational over $k$, hence is $R$-trivial. 
\end{proposition}
We need (Lemma 6.1, \cite{Th-4}), we include it here for convenience. 
\begin{lemma}\label{ext} Let $L/k$ be a cyclic cubic extension and $K/k$ a quadratic \'{e}tale extension and $\nu\in K^{\times}$ be an element with $N_K(\nu)=1$ and let $*:LK\rightarrow LK$ be the nontrivial $L$-automorphism of $LK$. Let $J$ denote the Tits process $J=J(LK,*, 1, \nu)$. Let $\rho$ be a generator of $Gal(L/k)$. Then $\rho$ extends to an automorphism of $J$.
  
\end{lemma}
   We now prove
  \begin{theorem}\label{aut-L-fixed} Let $A$ be a cyclic Albert division algebra and $L\subset A$ a cyclic cubic subfield. Then there exists $v\in L^{\times}$ such that $\text{Aut}(A^{(v)}/L)\subset R\text{\bf Aut}(A^{(v)})(k)$.  
  \end{theorem}
  \begin{proof} Let $S\subset A$ be a $9$-dimensional subalgebra such that $L\subset S$. By (6.4, \cite{P}) we may assume $S=J(LK,*, v, \nu)$ for $v\in L^{\times}$ and $\nu\in K^{\times}$ with $N_L(v)=N_K(\nu)=1$. We have
    $$S^{(v)}=J(LK,*, vv^{\#},N(v)\nu)=J(LK,*, 1, \nu).$$
    Let $\rho$ be a generator of $Gal(L/k)$. 
    By Lemma \ref{ext}, $\rho$ extends to an automorphism $\widetilde{\rho}$ of $S^{(v)}$. By (Prop. 3.4, \cite{Th-4}), $\widetilde{\rho}$ extends to an automorphism of $A^{(v)}$, which we denote abusing notation, also by $\widetilde{\rho}$. By Prposition \ref{aut-S-rtriv},
    $\widetilde{\rho}\in R\text{\bf Aut}(A^{(v)}, S^{(v)})\subset R\text{\bf Aut}(A^{(v)})(k)$. Now let $\phi\in\text{Aut}(A^{(v)}/L)$. Then $\widetilde{\rho}^{-1}\phi\notin\text{Aut}(A^{(v)}/L)$.
    By (\cite{Th-2}, Th. 4.1), $\widetilde{\rho}^{-1}\phi$ fixes a cubic subfield $M\subset A^{(v)}$ pointwise and $M\neq L$. The subalgebra $S'=<L,M>$ of $A^{(v)}$ generated by $L$ and $M$ is $9$-dimensional and is invariant under
    $\psi:=\widetilde{\rho}^{-1}\phi$. Hence by Proposition \ref{aut-S-rtriv}, we have
    $$\psi\in\text{Aut}(A^{(v)}, S')= R\text{\bf Aut}(A^{(v)}, S')(k)\subset R\text{\bf Aut}(A^{(v)})(k).$$
    Since $\widetilde{\rho}\in R\text{\bf Aut}(A^{(v)})(k)$, it follows that $\phi\in R\text{\bf Aut}(A^{(v)})(k)$. This proves that $\text{Aut}(A^{(v)}/L)\subset R\text{\bf Aut}(A^{(v})(k)$.
    \iffalse%%%%%%%%%%%%%%%%%%%%%%%%%%%%%%%%%%%%%%%%%%%%%%%%%%%%%%%%%%%%%%%%%%%%%%%%%%%%%%%%%%%%%%%%%%%%%%%%%%%%%%%
    Now, $L\subset (B,\sigma_w)_+$. Hence by (?), there exists $w'\in L^{\times}$ such that $Int(w')\sigma_w$ is a distinguished involution. Let $v:=w'w$. Then $v\in L^{\times}$ and $Int(w')\sigma_w=\sigma_v$ is distinguished. Since $A^{(v)}$ is the $w'$-isotope of $A^{(w)}$ with $w'\in L^{\times}$, by Lemma (?), we have $\text{\bf Aut}(A^{(v)}/L)\cong \text{\bf Aut}(A^{(w}/L)$.
    %%%%%%%%%%%%%%%%%%%%%%%%%%%%%%%%%%%%%%%%%%%%%%%%%%%%%%%%%%%%%%%%%%%%%%%%%%%%%%%%%%%%%%%%%%%%%%%%%%%%%%%%%%%%%%
    \fi
  \end{proof}
  \begin{corollary}\label{aut-L-stab} Let $A$ be a cyclic division algebra over a field $k$ and $L\subset A$ a cyclic cubic subfield. Then there exists $v\in L^{\times}$ such that $\text{Aut}(A^{(v)}, L)\subset R\text{\bf Aut}(A^{(v)})$. 
  \end{corollary}
  \begin{proof} Let $v\in L^{\times}$ be such that $\text{Aut}(A^{(v)}/L)\subset R\text{\bf Aut}(A^{(v)})(k)$. Let $\phi\in\text{Aut}(A^{(v)}, L)$. We may assume that $\phi|L=\rho\neq 1$. Then, with $\widetilde{\rho}$ as above, it follows that $\widetilde{\rho}^{-1}\phi$ fixes $L$ pointwise, hence belongs to $\text{Aut}(A^{(v)}/L)$ which is contained in $R\text{\bf Aut}(A^{(v)})(k)$. Also, $\widetilde{\rho}\in R\text{\bf Aut}(A^{(v)})$ by Proposition \ref{aut-S-rtriv}. Hence it follows that $\phi\in R\text{\bf Aut}(A^{(v)})(k)$. 
  \end{proof}
  We need the following result from (\cite{Th-4}) for our purpose here, we reproduce the proof for the sake of completeness.
 \begin{proposition}[\cite{Th-4}, Th. 6.4]\label{action} Let $A$ be an Albert division algebra over a field $k$ or arbitrary characteristic. Let $M\subset A$ be a cubic subfield and $\mathcal{S}_M$ denote the set of all $9$-dimensional subalgebras of $A$ that contain $M$. Then the group $\text{Str}(A,M)$ acts on $\mathcal{S}_M$.
 \end{proposition}
 \begin{proof} Let $\psi\in\text{Str}(A,M)$ and $S\in\mathcal{S}_M$ be arbitrary. Let $\phi\in\text{Aut}(A/S)$, $\phi\neq 1$. Then $A^{\phi}=S$. We have
   $$\psi\phi\psi^{-1}(1)=\psi(\psi^{-1}(1))=1,$$
   since $\psi^{-1}(1)\in M\subset S$ and $\phi$ fixes $S$ pointwise. Hence $\psi\phi\psi^{-1}\in\text{Aut}(A)$. It follows also that the fixed point subspace of $\psi\phi\psi^{-1}$ is precisely $\psi(S)$ and hence $\psi(S)$ is a subalgebra of $A$ of dimension $9$ and $M=\psi(M)\subset\psi(S)$, therefore $\psi(S)\in\mathcal{S}_M$. This completes the proof.
 \end{proof}
   \begin{proposition}\label{key-1} Let $A$ be an Albert division algebra over a field $k$ and $L\subset A$ be a cubic subfield. Let $\phi\in\text{Aut}(A)$. Then there exists $S\in\mathcal{S}_L$ and $\theta\in\text{Str}(A,S)$ such that $\theta\phi\in\text{Str}(A,L)$.
   \end{proposition}
\begin{proof} Consider the cubic subfield $M:=\phi(L)$ of $A$. If $M=L$, then $\phi\in\text{Aut}(A,L)\subset\text{Str}(A,L)$ and we are done by taking $\theta=1$ and $S$ as any $9$-dimensional subalgebra with $S\in\mathcal{S}_L$. So assume that $M\neq L$. Then $S:=<L,M>$, the subalgebra of $A$ generated by $L$ and $M$ is necessarily $9$-dimensional and applying (\cite{GP}, Th. 5.2.7) to the embeddings $\phi:L\rightarrow S$ and $\iota:L\rightarrow S$, there exists $\theta\in\text{Str}(A,S)$ and $w\in L^{\times}$ such that $\theta\phi(x)=wx$ for all $x\in L^{\times}$. In particular, $\theta\phi(L)=L$. Thus $\theta\phi\in\text{Str}(A,L)$.   
   \end{proof}
   \begin{proposition}\label{key-2} Let $A$ be a cyclic Albert division algebra over a field $k$ and $L\subset A$ be a cyclic cubic subfield. Let $S\in\mathcal{S}_L$ and $\psi\in\text{Str}(A,L)$.
     Assume that $S\cong\psi(S)$. Then there exists $\phi\in\text{Aut}(A,L)$ such that $\phi^{-1}\psi\in \text{Str}(A,S)$.
   \end{proposition}
   \begin{proof} We have $S=(B,\sigma)_+$ for a degree $3$ central division algebra $B$ over a quadratic \'{e}tale extension $K/k$ and $\sigma$ a unitary involution on $B$.
      Since $\psi\in\text{Str}(A,L)$ and $S\cong\psi(S)$, it follows that $\psi(S)$ is isomorphic to an isotope of $S$ by an element of $L^{\times}$. We may assume, by (\cite{PR7}, Prop. 3.8),  $\psi((B,\sigma)_+)\cong(B,\sigma_v)_+$ for some $v\in L$ with $N_L(v)=1$, since $\psi(L)=L$. Hence we have $(B,\sigma)_+\cong (B,\sigma_v)_+$ and in particular, the involutions $\sigma$ and $\sigma_v$ on $B$ are conjugate. By (\cite{KMRT}, Cor. 19.31),  there exists $w\in (LK)^{\times}$ and $\lambda\in k^{\times}$ such that $v=\lambda ww^*$, where $*$ is the nontrivial $k$-automorphism of $K/k$, extended $L$-linealy to an automorphism of $LK$.

     By (\cite{PT}, 1.6), the inclusion $L\subset (B,\sigma)_+$ extends to an isomorphism $\delta:J(LK, *, u, \mu)\rightarrow (B,\sigma)_+$ for a suitable admissible pair $(u,\mu)\in L^{\times}\times K^{\times}$. Moreover, $L\subset (B,\sigma_v)_+$, hence there is an isomorphism $\eta: J(LK, *, u', \mu')\rightarrow\psi(S)=(B,\sigma_v)_+$ for suitable $(u', \mu')$ extending the inclusion $L\hookrightarrow (B,\sigma_v)_+$. Since $\delta$ and $\eta$ are identity on $L$, it follows that $J(LK, *, u', \mu')$ is the $v$-isotope of $J(LK, *, u, \mu)$. Hence, by (\cite{PR7}, Prop. 3.9), we may assume $(u', \mu')=(uv^{\#}, N_L(v)\mu)$. We have therefore
     $$u'=uv^{\#}=u(\lambda w w^*)^{\#}=u\lambda^2 w^{\#}(w^{\#})^*=uw_0w_0^*,~\mu'=N_L(v)\mu,$$
     where $w_0=\lambda w^{\#}$. We have an isomorphism (see \cite{PR7}, Prop. 3.7)
     $$\phi_1: J(LK, *, uv^{\#}, N_L(v)\mu)=J(LK, *, uw_0w_0^*, \mu)\rightarrow J(LK, *, u, N_{LK}(w_0)^{-1}\mu).$$
     We have
     $$N_{LK}(w_0)\overline{N_{LK}(w_0)}=N_{LK}(\lambda w^{\#})\overline{N_{LK}(\lambda w^{\#})}=(\lambda^3 N_{LK}(w)^2)\lambda^3\overline{N_{LK}(w)^2)}$$
     $$=(\lambda^3 N_{LK}(w)\overline{N_{LK}(w)})^2=N_L(v)^2=1.$$
     Hence, by (\cite{PT}, Lemma 4.5), there exists $w_1\in LK$ with $w_1 w_1^*=1$ and $N_{LK}(w_1)=N_{LK}(w_0)$.
     The map 
     $$\phi_2:J(LK, *, u, N_{LK}(w_0)^{-1}\mu)\rightarrow J(LK, *, u, \mu)$$
     given by $\phi_2((l,x))=(l, xw_1^{-1})$ is an isomorphism by (\cite{PR7}, Prop. 3.7). Hence the composite $\phi_1^{-1}\phi_2^{-1}: J(LK, *, u, \mu)\rightarrow J(LK, *, u', \mu')$ is an isomorphism which maps $L$ to $L$. We have therefore the isomorphism $\phi:=\eta(\phi_1^{-1}\phi_2^{-1})\delta^{-1}:S\rightarrow \psi(S)$, which satisfies $\phi(L)=L$ and $\phi^{-1}\psi(S)=S$, where $\phi$ denotes any extension of $\phi$ as above, to an automorphism of $A$, which is possible by (\cite{P-S-T1} and \cite{P}). This completes the proof. 
   \end{proof}
   We need the following intermediate result.
   \begin{proposition}\label{inter} Let $A$ be a cyclic Albert division algebra over a field $k$ of arbitrary characteristic with $f_5(A)=0$. Let $L\subset A$ be a cubic cyclic subfield and assume $A$ contains a second Tits process $J(LK,*,1,\nu)$ as a subalgebra. Then $\text{\bf Aut}(A)$ is $R$-trivial.
   \end{proposition}
   \begin{proof} Let $\rho$ be a generator of $Gal(L/k)$ and $\widetilde{\rho}$ denote the extension of $\rho$ to an automorphism of $J(LK,*,1,\nu)$ and its firther extension to an automorphism of $A$ as well. By Proposition \ref{aut-S-rtriv}, $\widetilde{\rho}\in R\text{\bf Aut}(A)(k)$.

     Now let $\phi\in\text{Aut}(A)$. First assume $\phi$ fixes $L$ pointwise. Consider the automorphism $\psi:=\widetilde{\rho}\phi\in\text{Aut}(A)$. Clearly $\psi(L)=L$ and $\psi\notin\text{Aut}(A/L)$. By (\cite{Th-2}, Th. 4.1)), there is a cubic subfield $M\subset A$ such that $\psi$ fixes $M$ pointwise. Then $M\neq L$ and the subalgebra $S:=<L,M>$ of $A$ generated by $L$ and $M$ is $9$-dimensional. Clearly $\psi(S)=S$. Hence by Proposition \ref{aut-S-rtriv}, $\psi=\widetilde{\rho}\phi\in R\text{\bf Aut}(A)(k)$. Since $\widetilde{\rho}\in R\text{\bf Aut}(A)(k)$, we have $\phi\in R\text{\bf Aut}(A)(k)$.

     Now assume $\phi\in\text{\bf Aut}(A,L)$ and $\phi\notin\text{Aut}(A/L)$. Without loss of generality, we assume $\phi|L=\rho$. Then $\widetilde{\rho}^{-1}\phi\in\text{Aut}(A/L)\subset R\text{\bf Aut}(A)(k)$. It follows that $\phi\in R\text{\bf Aut}(A)(k)$. Finally assume $M:=\phi(L)\neq L$. By Proposition \ref{key-1}, there exists $S\in\mathcal{S}_L$ and $\theta\in\text{Str}(A,S)$ such that $\theta\phi\in\text{Str}(A,L)$. Hence $\phi^{-1}\theta^{-1}(L)=L$. Let $\gamma:=\phi^{-1}\theta^{-1}$. Then $\gamma\in\text{Str}(A,L)$ and since $\theta(S)=S$, we have $\gamma(S)=\phi^{-1}(S)\cong S$. By Proposition \ref{key-2}, there exists $\eta\in\text{Aut}(A,L)$ such that $\eta^{-1}\gamma\in\text{Str}(A,S)$. We therefore have,
     $$S=\eta^{-1}\gamma(S)=\eta^{-1}\phi^{-1}\theta^{-1}(S)=\eta^{-1}\phi^{-1}(S).$$
     Hence $\eta^{-1}\phi^{-1}\in\text{Aut}(A, S)$. Since, by Proposition \ref{aut-S-rtriv} $\text{Aut}(A,S)\subset R\text{\bf Aut}(A)(k)$  and $\eta\in\text{Aut}(A,L)\subset R\text{\bf Aut}(A)(k)$, it follows that $\phi\in\text{\bf Aut}(A)(k)$.

     Finally, if $N$ is any field extension of $k$, we have $\text{\bf Aut}(A)(N)=\text{Aut}(A\otimes N)$. If $A\otimes N$ is a division algebra, by considering the cyclic subfield $L\otimes N\subset A\otimes N$ and the Tits process subalgebra $J(L_NK,*,1,\nu)\subset A\otimes N$, by exactly same arguments as above, we get $\text{Aut}(A_N)\subset R\text{\bf Aut}(A)(N)$. On the other hand, when $A_N$ is reduced, the hypothesis $f_5(A)=0$ implies $A_N$ has nonzero nilpotents and hence the abstract group $\text{Aut}(A_N)$ is simple (Thm. 14, \cite{J-iv}). Since $\text{\bf Aut}(A)$ has rank $2$-subgroups defined over $k$, it follows that $R\text{\bf Aut}(A)(N)\neq \{1\}$. Hence, simplicity implies $\text{Aut}(A_N)=R\text{\bf Aut}(A)(N)$. This completes the proof that $\text{\bf Aut}(A)$ is $R$-trivial. 
     
     \end{proof}
  We now proceed to prove our main results.
  \begin{theorem}\label{self-isotopic} Let $A$ be an Albert division algebra over a field $k$ of arbitrary characteristic whose all isotopes are isomorphic. Then the algebraic group $\text{\bf Aut}(A)$ is $R$-trivial.
  \end{theorem}
  \begin{proof} By (\cite{Th-3}), since all isotopes of $A$ are isomorphic to $A$, it follows that $A$ contains a cyclic cubic subfield $L$. Also, if $A_{red}$ denotes the reduced model of $A$ over $k$ (see \cite{PR}), then the hypothesis on $A$ implies $A_{red}$ contains nonzero nilpotents and $f_5(A)=0$. By Corollary \ref{aut-L-stab}, there is $v\in L^{\times}$ such that $\text{Aut}(A^{(v)}, L)\subset R\text{\bf Aut}(A^{(v)})$. Hence $\text{Aut}(A,L)\subset R\text{\bf Aut}(A)(k)$. Let $S\subset A$ be a $9$-dimensional subalgebra containing $L$. We may write $S=(B,\sigma)_+$ for a suitable degree $3$ central division algebra $B$ with a unitary involution over $K:=Z(B)$. By (\cite{P}, 6.4), we may assume $S=J(LK, *, w, \nu)$ for suitable admissible pair $(w,\nu)$ with $N_L(w)=N_K(\nu)=1$. By hypothesis $A^{(w)}\cong A$ and $A^{(w)}\supset S^{(w)}$. Also $S^{(w)}=J(LK,*, w, \nu)^{(w)}\cong J(LK,*, 1, \nu)$. Hence we may assume that $A$ contains $J(LK, *, 1, \nu)$ as a subalgebra. Thus $A$ satisfies the hypothesis of Proposition \ref{inter} and the result now follows, i.e., $\text{\bf Aut}(A)$ is $R$-trivial.

  \end{proof}
  \begin{corollary} Let $A$ be an Albert algebra arising from the first Tits construction. Then $\text{\bf Aut}(A)$ is $R$-trivial.
  \end{corollary}
  \begin{proof} $A$ is either split, in which case $\text{\bf Aut}(A)$ is split, hence $R$-trivial, or $A$ is division and all isotopes of $A$ are isomorphic, by (\cite{PR2}). Hence the theorem applies to $A$ and $\text{\bf Aut}(A)$ is $R$-trivial. 
    \end{proof} 
    When the base field contains cube roots of unity, we can prove a better result. First recall that under this hypothesis on the base field $k$, by (\cite{PF}), any Albert division algebra over $k$ is cyclic. We have
    \begin{theorem} Let $k$ be a field of arbitrary characteristic containing cube roots of unity. Let $A$ be an Albert division algebra over $k$. Then there exists $v\in A$ such that the group $\text{\bf Aut}(A^{(v)})$ is $R$-trivial.
    \end{theorem}
    \begin{proof} Let $k$ and $A$ be as in the hypothesis, then $A$ is cyclic. Let $L\subset A$ be a cubic cyclic subfield. Then there exists a $9$-dimensional subalgebra $S\subset A$ with $L\subset S$. We may assume that $S=(B,\sigma)_+$ for a central simple algebra $(B,\sigma)$ with a unitary involution over its center $K$, which is a quadratic \'{e}tale extension of $k$. By (1.6, \cite{PT}; 6.4, \cite{P}), we can extend the inclusion $L\subset (B,\sigma)_+$ to an isomorphism of the Tits process $J(LK,*, v, \nu)$ with $(B,\sigma)_+$ for some $v\in L$ with $N_L(v)=1=N_K(\nu)$. We may therefore assume that the Tits process $J(LK,*, v,\nu)\subset A$. Passing to the $v$-isotope $A^{(v)}$ of $A$, we have $J(LK,*, v, \nu)^{(v)}\subset A^{(v)}$ as $v\in L\subset J(LK,*,v,\nu)$. We therefore have
      $$J(LK,*,v,\nu)^{(v)}\cong J(LK, *, vv^{\#}, N_L(v)\nu)=J(LK, *, 1, \nu),$$
      using $N_L(v)=vv^{\#}=1$. Hence $J(LK, *, 1, \nu)\subset A^{(v)}$. We next prove that $f_5(A^{(v)})=0$. We may write $J(LK,*,1,\nu)=(C,\tau)_+$ for a suitable degree $3$ central division algebra $C$ with a unitary involution $\tau$ over its center. First assume that characteristic of $k$ is not $3$. Then the hypthesis on $k$ allows a Kummer element $w\in L$ with $T_L(w)=T_L(w^2)=0$. Then $w^3=:a\in k^{\times}$. Hence the involution $\tau_w=int(w)\tau$ is distinguished. We have,
      $$(C,\tau_w)_+\cong (C,\tau)_+^{(w)}=J(LK,*, 1, \nu)^{(w)}\cong J(LK,*, w^2, a\nu)\cong J(LK, *, 1, \nu).$$
      It follows from this that $f_5(A^{(v)})=0$ (see 40.7, \cite{KMRT}). Now assume characteristic of $k$ is $3$. Then, by (2.9, 2.10, \cite{P}), it follows that $\tau$ is a distinguishe involution on $C$ and since $(C,\tau)_+\subset A^{(v)}$, again we have $f_5(A^{(v)})=0$.

      We have therefore shown that under the hypothesis on $k$, there is $v\in L$ such that $J(LK,*, 1, \nu)\subset A^{(v)}$ and $f_5(A^{(v)})=0$. The result now follows from Proposition \ref{inter}.
      \iffalse%%%%%%%%%%%%%%%%%%%%%%%%%%%%%%%%%%%%%%%
      We may therefore assume, to begin with, that $A$ is clyclic with $L\subset A$ cyclic cubic subfield, $J(LK,*, 1, \nu)\subset A$ and $f_5(A)=0$. We will prove that $\text{\bf Aut}(A)(M)\subset R\text{\bf Aut}(M)$ for any field extension $M$ of $k$. This follows by an argument exactly as in the proof of Theorem \ref{self-isotopic}. 

      %%%%%%%%%%%%%%%%%%%%%%%%%%%%%%%%%%%%%%%%%%%%%%%%%%%%%%%%%%%%%%%%%%%%%%%%%%%%%%%%%%%%%%%%%%%%
      \fi
      \end{proof}

 \noindent
 {\bf Acknowledgement} Part of this work was done while the author was visiting Linus Kramer of Mathematics Institute, Uni-M\"{u}nster in the summer of 2019. We thank Linus Kramer for his support and hospitality during this visit. The author also thanks Holger Petersson for his interest in this work and some fruitful discussions at Hagen during the above period. We are grateful to the referee for his/her efforts in reading the manuscript carefully and suggesting changes which have resulted in better readability of the paper. We thank V. Nandagopal of School of Mathematics, TIFR for providing us technical support during the preparation of this manuscript. 

\vskip5mm
Indian Statistical Institute, Stat-Math.-Unit, 8th Mile Mysore Road, Bangalore-560059, India.
\center email: maneesh.thakur@gmail.com

\end{document}